\documentclass{article}

\title{Topological mirror symmetry for rank two character varieties of surface groups}
\author{Mirko Mauri}
\usepackage[utf8]{inputenc}
\usepackage[T1]{fontenc}
\usepackage{fixltx2e}
\usepackage{amsthm}
\usepackage{mathrsfs}

\usepackage{mathtools}
\usepackage{graphicx}
\usepackage{longtable}
\usepackage{float}
\usepackage{wrapfig}
\usepackage{rotating}
\usepackage{amsmath}
\usepackage{textcomp}
\usepackage{marvosym}
\usepackage{wasysym}
\usepackage{amssymb}
\usepackage{hyperref}
\usepackage{tikz}
\usepackage{tikz-cd}
\usepackage{float}
\usepackage{stmaryrd}
\usepackage{enumerate}
\usepackage{caption}
\usepackage{multirow}
\usepackage{fancyhdr}

\usetikzlibrary{matrix}
\usetikzlibrary{shapes.multipart}
\usetikzlibrary{arrows}
\usepackage{amsthm,amsmath,amsfonts,amscd}
\usepackage{tikz-3dplot}
\usetikzlibrary{decorations.markings}
\usetikzlibrary{arrows}
\usepackage{parskip}
\allowdisplaybreaks

\newcommand{\CC}{\mathbb{C}}

\newcommand{\ZZ}{\mathbb{Z}}

\newcommand{\Gr}{\operatorname{Gr}}

\newcommand{\Gl}{\operatorname{GL}}
\newcommand{\Sl}{\operatorname{SL}}
\newcommand{\PGl}{\operatorname{PGL}}

\newcommand{\MB}{M_{\mathrm{B}}}
\newcommand{\MDR}{M_{\mathrm{DR}}}
\newcommand{\MDol}{{M}_{\mathrm{Dol}}}

\newcommand{\Hom}{\operatorname{Hom}}
\newcommand{\codim}{\operatorname{codim}}

\newcommand{\Pic}{\operatorname{Pic}}

\usepackage{verbatim}

\begingroup
\makeatletter
\@for\theoremstyle:=definition,remark,plain\do{%
\expandafter\g@addto@macro\csname th@\theoremstyle\endcsname{%
\addtolength\thm@preskip\parskip
}%
}
\endgroup
\usepackage{graphicx}
\usepackage[capitalise]{cleveref}
\newtheorem{thm}{Theorem}[section]

\newtheorem{conj}[thm]{Conjecture}

\theoremstyle{definition}

\newtheorem{rmk}[thm]{Remark}
\crefname{thm}{Theorem}{Theorems}
\Crefname{thm}{Theorem}{Theorems}
\Crefname{thm}{Theorem}{Theorems}
\Crefname{thm}{Theorem}{Theorems}
\crefname{lem}{Lemma}{Lemmas}
\Crefname{lem}{Lemma}{Lemmas}
\crefname{Conjecture}{Conjecture}{Conjectures}
\Crefname{Conjecture}{Conjecture}{Conjectures}
\crefname{Corollary}{Corollary}{Corollaries}
\Crefname{Corollary}{Corollary}{Corollaries}
\crefname{Claim}{Claim}{Claims}
\Crefname{Claim}{Claim}{Claims}
\crefname{Proposition}{Proposition}{Propositions}
\Crefname{Proposition}{Proposition}{Propositions}
\crefname{Remark}{Remark}{Remarks}
\Crefname{Remark}{Remark}{Remarks}
\crefname{Definition}{Definition}{Definitions}
\Crefname{Definition}{Definition}{Definitions}
\crefname{Example}{Example}{Examples}
\Crefname{Example}{Example}{Examples}
\crefname{Exercise}{Exercise}{Exercises}
\Crefname{Exercise}{Exercise}{Exercises}

\newtheoremstyle{plain2}    
   {}            
   {}            
   {\itshape}    
   {}            
   {\bfseries}   
   {.}           
   {5pt plus 1pt minus 1pt}  
   {{\thmnumber{#1} \thmname{#2}{\thmnote{ (#3)}}}}          

\begin{document}
\maketitle
\begin{abstract}
    The moduli spaces of flat $\mathrm{SL}_2$- and $\mathrm{PGL}_2$-connections are known to be singular SYZ-mirror partners. We establish the equality of Hodge numbers of their intersection (stringy) cohomology. In rank two, this answers a question raised by Tamás Hausel in Remark 3.30 of "Global topology of the Hitchin system".
\end{abstract}

\vspace{0.5 cm}

Let $C$ be a compact Riemann surface of genus $g$ with base point $c \in C$, and $G$ be either $\Sl_r$ or $\PGl_r$. We study the following moduli spaces (cf \cite{Simpson1994}):
\begin{itemize}
    \item the \textbf{de Rham} moduli space of principle flat $G$-bundles on $C \setminus c$ with holonomy $e^{2\pi i d/r}$ around $c$; 
    \item the \textbf{Dolbeault} moduli space of semistable $G$-Higgs bundles of degree $d$, i.e.\ semistable pairs $(E, \phi)$ consisting of a principal $G$-bundle $E$ of degree $d$ and a section $\phi \in H^0(C, \operatorname{ad}(E)\otimes K_C)$, where $K_C$ is the canonical bundle;
    \item the \textbf{Betti} moduli space parametrising $G$-representations of the fundamental group of $C \setminus c$ with monodromy $e^{2\pi i d/r}$ around $c$.
\end{itemize}
These moduli spaces are denoted respectively $\MDR^d(C, G)$, $\MDol^d(C, G)$ and $\MB^d(C, G)$. For convenience, we simply write $M(C, G)$ when we refer indifferently to  $\MDol^0(C,G)$, $\MDR^0(C, G)$ or $\MB^0(C,G)$.

In \cite{HauselThaddeus03}, Hausel and Thaddeus showed that the de Rham moduli spaces $\MDR^d(C, \Sl_r)$ and $\MDR^d(C, \PGl_r)$ are mirror partners in the sense of Strominger--Yau--Zaslow mirror symmetry. According to the general mirror symmetric framework, it is reasonable to expect a symmetry between their Hodge numbers. 

Hausel and Thaddeus conjectured the equality of the stringy E-polynomials
\begin{equation}\label{eq:topologicalmirrorsymmetry}
     E_{\mathrm{st}}^{B^e}(\MDR^d(C, \Sl_r))=E_{\mathrm{st}}^{\hat{B}^d}(\MDR^e(C, \PGl_r)),
\end{equation}
for $(d,r)=(e,r)=1$, and they prove it for $r=2,3$. The conjecture is now a theorem due to \cite{GWZ20} or \cite{MaulikShen2020I}.

In \cite[Remark 3.30]{Hausel13} Hausel asked  what cohomology theory we should compute on $\MDR^d(C, \Sl_r)$, with $(d,r)\neq 1$, to accomplish the agreement \eqref{eq:topologicalmirrorsymmetry}. 

We propose to use intersection cohomology. As first piece of evidence, we show the topological mirror symmetry conjecture in rank two, and degree zero, i.e.\ when we turn off the B-fields\footnote{See \cite[\S 4]{HauselThaddeus03} for a definition of gerbe or B-field.} $B$ and $\hat{B}$. 

\begin{thm}[Topological mirror symmetry in rank two and degree zero]\label{mainthm} The intersection E-polynomial of $M(C, \Sl_2)$ equals the stringy intersection E-polynomial of $M(C, \PGl_2)$
\[IE(M(C, \Sl_2)) = IE_{\mathrm{st}}(M(C, \PGl_2)).\]
\end{thm}
The refinements of the Hausel--Thaddeus conjecture postulated in \cite[Conjecture 3.27]{Hausel13} and \cite[Conjecture 5.9]{Hausel13} also hold true in rank two and degree zero, as long as we consider their intersection cohomology analogues; see \cref{thm:topologicalmirrorsymmetryrank2} and \cref{thm:perversetopologicalmirrorsymmetryrank2}.

\section{Intersection stringy E-polynomial}\label{sec:intersectionstringyE-pol}
The intersection cohomology of a complex variety $X$ with compact support, middle perversity and rational coefficients is denoted by $IH^*_c(X)$. Recall that $IH^*_c(X)$ carries a canonical mixed Hodge structure, and so we can define the intersection E-polynomial of $X$ as
\begin{equation}\label{defn:IEpol}
    IE(X) \coloneqq \sum_{r,s,d}(-1)^d \dim ( \Gr^W_{r+s} IH^d_{c}(X, \CC))^{r,s} u^rv^s.
\end{equation}
Suppose that $X$ is endowed with the action of a finite abelian group $\Gamma$, and denote the group of characters of $\Gamma$ by $\hat{\Gamma}$. The intersection cohomology of $X$ decomposes under the action of $\Gamma$ into isotypic components:
\[IH^*_c(X)= \bigoplus_{\kappa \in \hat{\Gamma}} IH^*(X)_{\kappa}.\]
Then, if we pose
\begin{align*}
    IE(X)_{\kappa} & \coloneqq \sum_{r,s,d}(-1)^d \dim ( \Gr^W_{r+s} IH^d_{c}(X, \CC)_{\kappa})^{r,s} u^rv^s,
\end{align*}
we obtain $IE(X)= \sum_{\kappa \in \hat{\Gamma}} IE(X)_{\kappa}$. 

Define also the intersection stringy E-polynomial by
\[IE_{\mathrm{st}}(X) \coloneqq \sum_{\gamma \in \Gamma} IE(X_{\gamma}/\Gamma; u,v)(uv)^{F(\gamma)},\]
where 
\begin{itemize}
    \item $X_{\gamma}$ is the fixed-point set of $\gamma \in \Gamma$.
    \item $F(\gamma)$ is the Fermionic shift, defined as $F(\gamma)=\sum_j w_j$, where $\gamma$ acts on the normal bundle of $X_{\gamma}$ in $X$ with eigenvalues $e^{2\pi i w_j}$ with $w_j \in (0,1)$.
\end{itemize}

\section{Topological mirror symmetry}
Let $\Gamma \coloneqq \mathrm{Pic}^0(C)[r]\simeq (\ZZ/r\ZZ)^{2g}$ be the group of $r$-torsion line bundles over the compact Riemann surface $C$ of genus $g$, endowed with the canonical flat connection. The group $\Gamma$ acts by tensorisation on $\MDol^d(C, \Sl_r)$ and $\MDR^d(C, \Sl_r)$. Via the non-abelian Hodge correspondence, the action corresponds to the algebraic action of the characters
$\Gamma \simeq \Hom(\pi_1(C), \pm 1)$ which acts on $\MB^d(C, \Sl_r)$ by multiplication. The quotient of $M^d(C, \Sl_r)$ by the action of $\Gamma$ is isomorphic to $M^d(C, \PGl_r)$.

We identify $w\colon \Gamma \to \hat{\Gamma}$ through Poincar\'{e} duality (also known as Weil pairing)
\[\Gamma \times {\Gamma} \simeq H_1(C, \ZZ/r\ZZ)\times H_1(C, \ZZ/r\ZZ) \to \ZZ/r\ZZ.\] 
\begin{conj}[Topological mirror symmetry in degree zero]\label{conj:topmirr} For $\kappa \in \hat{\Gamma}$ we have
\begin{equation}\label{eq:refinedtopologicalmirrorsymmetry}
    IE(M(C, \Sl_r))_{\kappa} =IE(M(C, \Sl_r)_{\gamma}/\Gamma; u,v)(uv)^{F(\gamma)}
\end{equation}
where $\gamma = w(\kappa)$. In particular, we obtain
\[IE(M(C, \Sl_r))=IE_{\mathrm{st}}(M(C, \PGl_r)).\]
\end{conj}

\begin{thm}[\cref{mainthm}]\label{thm:topologicalmirrorsymmetryrank2}
\cref{conj:topmirr} holds for $r=2$.
\end{thm}
\begin{proof}
Without loss of generality we can suppose $\gamma\neq 0$, or equivalently $\kappa \neq 1$. Indeed, \[IE(M(C, \Sl_r))_{1} =IE(M(C, \Sl_r)/\Gamma; u,v);\]
see for instance the proof of \cite[Proposition 3]{GottscheSoergel93}. 

For any $\gamma \in \Gamma \setminus \{0\}$, we have an associated 2-torsion line bundle $L_{\gamma}$. Consider the \'{e}tale double cover $\pi_\gamma\colon C_{\gamma} \to C$ consisting of the square root of a non-zero section of  $L_{\gamma}^{\otimes 2}\simeq \mathcal{O}_C$ in the total space of $L_{\gamma}$, and let $\iota$ be its deck transformation. 

For any $L \in M(C_{\gamma}, \Gl_1)$, the rank-two vector bundle $L \oplus \iota^* L$ is a $\iota$-invariant object in $M(C_{\gamma}, \Gl_2)$, which descends to an object $L_{\iota} \in M(C, \Gl_2)$. Hence, the pushforward morphism
\begin{align*}
    \pi_{\gamma, *}\colon M(C_{\gamma}, \Gl_1) \to M(C, \Gl_2), \quad L \mapsto L_{\iota},
\end{align*}
descends to a $\Gamma$-invariant embedding
\[j\colon M(C_{\gamma}, \Gl_1)/\ZZ/2\ZZ \hookrightarrow M(C, \Gl_2).\]
The determinant map $\det\pi_{\gamma, *}$ can be identified with the norm map
\begin{align*}
    \mathrm{Nm}_{C_{\gamma}/C}\colon M(C_{\gamma}, \Gl_1)/\ZZ/2\ZZ \to M(C, \Gl_1), \quad L \mapsto L\otimes \iota^*L,
\end{align*}
Therefore, the fixed-point set $M(C, \Sl_2)_{\gamma}$ admits the following geometric characterization: 
\[M(C, \Sl_2)_{\gamma} = \mathrm{Im}j \cap M(C, \Sl_2) \simeq \ker \mathrm{Nm}_{C_{\gamma}/C}^{\circ},\]
where the last term is the connected component of $\mathrm{Nm}_{C_{\gamma}/C}^{-1}(\mathcal{O}_C)$ containing $\mathcal{O}_{C_{\gamma}}$. 

On the Dolbeault side, $\MDol(C, \Sl_2)_{\gamma}$ is isomorphic to the quotient by $\ZZ/2\ZZ$ of the cotangent bundle of an abelian variety of dimension $g-1$, as $\MDol(C_{\gamma}, \Gl_1)$ is isomorphic to $T^*\Pic^0(C_{\gamma})$; see also the proof of \cref{thm:perversetopologicalmirrorsymmetryrank2}. On the Betti side, we have 
\[\MB(C_{\gamma}, \Gl_1) \simeq (\CC^*)^{4g-2}, \text{ and so } \MB(C, \Sl_2)_{\gamma} \simeq (\CC^*)^{2g-2}/{\ZZ/2\ZZ}.\]

The involution defining the $\ZZ/2\ZZ$-quotient is the inverse of the group law.

Since the $\Gamma$-module $IH^*_c(M(C, \Sl_2))$ is a direct sums of copies of the trivial and regular representations by \cite[Remark 4.4]{Mauri20}, we have
\[IE(M(C, \Sl_2))_{\kappa} = IE(M(C, \Sl_2))_{\kappa'}\]
for any $\kappa, \kappa' \in \hat{\Gamma}\setminus \{1\}$. Thanks to \cite[Corollary 1.11, Equations (23) and (25)]{Mauri20} we have
\begin{align}
    IE(\MDol(C, \Sl_2))_{\kappa}  = & \frac{1}{2} (uv)^{3g-3} ((u+1)^{g-1}(v+1)^{g-1} +(u-1)^{g-1}(v-1)^{g-1}) \nonumber\\
     = & IE(\MDol(C, \Sl_2)_{\gamma}/\Gamma)(uv)^{2g-2},\nonumber\\
    IE(\MB(C, \Sl_2))_{\kappa}  = & \frac{1}{2}  (uv)^{2g-2} ((uv+1)^{2g-2}+(uv-1)^{2g-2}) \label{eq:intersectioncohomokappa}\\
    = & IE((\CC^*)^{2g-2}/{\ZZ/2\ZZ})(uv)^{2g-2} \nonumber\\
    = & IE(\MB(C, \Sl_2)_{\gamma}/\Gamma)(uv)^{2g-2}. \nonumber
\end{align}
Note that the Fermionic shift $F(\gamma)$ equals half of the codimension of $\MDol(C, \Sl_2)_{\gamma}$ in $\MDol(C, \Sl_2)$, since $\gamma$ is an involution. Hence, for $\gamma \neq 0$ we have indeed
\[F(\gamma)= \frac{1}{2}\codim \MDol(C, \Sl_2)_{\gamma} = 2g-2.\]

Finally, the same argument of \cite[\S 6]{HauselThaddeus03}, together with \cite[Theorem 3.2]{FelisettiMauri2020}, implies that Conjecture \eqref{eq:refinedtopologicalmirrorsymmetry} for the Dolbeault moduli spaces yields \eqref{eq:refinedtopologicalmirrorsymmetry} for the de Rham moduli spaces.
\end{proof}
\begin{rmk}[Failure of topological mirror symmetry for ordinary cohomology] In general the equality \eqref{eq:refinedtopologicalmirrorsymmetry} fails for ordinary cohomology. For instance, for $\kappa \neq 1$, $\gamma=w(\kappa)$ and $q=uv$ we have
\begin{align*}
    E(M_B(C, \Sl_2))_{\kappa} & =\frac{1}{2}q^{2g-2}((q+1)^{2g-2}+(q-1)^{2g-2}-2)\\ 
    & \neq \frac{1}{2}q^{2g-2}((q+1)^{2g-2}+(q-1)^{2g-2})= E(M_B(C, \Sl_2)_{\gamma}/\Gamma)q^{F(\gamma)},
\end{align*}
where the first equality follows from \cite[Theorem 2]{MartinezMunoz16} or \cite[Theorem 1.3]{BaragliaHekmati17}, together with \cite[Remark 4.3]{Mauri20}, while the last equality comes from \eqref{eq:intersectioncohomokappa}, since $M_B(C, \Sl_2)_{\gamma}/\Gamma$ has only quotient singularities. This shows that there is a non-negligible contribution of the singularity of $M(C, \Sl_r)$ to the agreement \eqref{eq:refinedtopologicalmirrorsymmetry} of Hodge numbers.
\end{rmk}
\begin{rmk}
The proof of \cref{thm:topologicalmirrorsymmetryrank2} relies on the computation of $\sum_{\kappa \neq 1} IE(\newline M(C, \Sl_2))_{\kappa} $ in \cite[Corollary 1.11]{Mauri20}, and ultimately on the explicit construction of a desingularization of $M(C, \Sl_2)$ in \cite[\S 3]{Mauri20}. To the best of the author's knowledge, this is not available in higher rank, and so it is unclear if the arguments above extend in higher rank. However, remarkable progress in this direction have been made in \cite{MaulikShen2020I} and \cite{MaulikShen2020II}.
\end{rmk}
\color{red}
\color{black}

\section{Perverse topological mirror symmetry}
The intersection cohomology of $\MDol(C, \Sl_r)$ and $\MDol(C, \Sl_2)_{\gamma}$ are filtered by the weight filtration $W$ of Deligne's mixed Hodge structure, and by the perverse filtration $P$ associated to the Hitchin fibrations
\begin{align*}
    \chi  \colon \MDol(C, \Sl_r) & \to \Lambda \coloneqq \bigoplus^n_{i=2} H^0(C, K^{\otimes i}_C) \\
    \chi_{\gamma}\coloneqq \chi|_{\MDol(C, \Sl_r)_{\gamma}} \colon \MDol(C, \Sl_r)_{\gamma} & \to \Lambda_{\gamma}\coloneqq \mathrm{Im}(\chi|_{\MDol(C, \Sl_r)_{\gamma}}) \subseteq \Lambda,
\end{align*}
which assigns to the Higgs bundle $(E, \phi)$ the characteristic polynomial of $\phi$; see \cite[\S 2.2]{Mauri20} for a brief account on the perverse filtration.

Define the perverse intersection E-polynomial
\[PIE(\MDol(C, \Sl_r);u,v,q)\coloneqq \sum_{r,s,d}(-1)^d \dim ( \Gr^W_{r+s} \Gr^P_{k}IH^d_{c}(X, \CC))^{r,s} u^rv^sq^k\]
and the stringy perverse intersection E-polynomial
\[PIE_{\mathrm{st}}(\MDol(C, \PGl_r); u,v,q)\coloneqq \sum_{\gamma \in \Gamma} PIE(\MDol(C, \Sl_r)_{\gamma}/\Gamma; u,v)(uvq)^{F(\gamma)}.\]
By Definition \eqref{defn:IEpol} and the last paragraph of the proof of \cref{thm:topologicalmirrorsymmetryrank2}, we have 
\[PIE(\MDol(C, \Sl_r);u,v,1) = IE(\MDol(C, \Sl_r);u,v)=IE(\MDR(C, \Sl_r);u,v).\]
Further, Relative Hard Lefschetz \cite[Theorem 2.1.1]{deCataldoMigliorini05} implies
\[PIE(\MDol(C, \Sl_r);u,v,q)=(uvq)^{\dim}PIE\bigg(\MDol(C, \Sl_r);u,v,\frac{1}{uvq}\bigg),\]
where $\dim = 2(r^2-1)(g-1)$.

We conjecture the exchange of the perverse Hodge numbers.
\begin{conj}[Perverse topological mirror symmetry in degree zero]\label{conj:pervtopmirr}
\[PIE(\MDol(C, \Sl_r);u,v,q)=(uvq)^{\dim}PIE_{\mathrm{str}}\bigg(\MDol(C, \PGl_r);u, v, \frac{1}{uvq}\bigg).\]
\end{conj}
For $q=1$, \cref{conj:pervtopmirr} specialises to  
\[IE(\MDol(C, \Sl_r)) = IE_{\mathrm{st}}(\MDol(C, \PGl_r)).\]
Further, the PI=WI conjecture \cite[Question 4.1.7]{deCataldoMaulik2018} would imply 
\[PIE(\MDol(C, \Sl_r);1,1,q) = IE(\MB(C, \Sl_r); q),\]
and together with \cref{conj:pervtopmirr} it would give
\[IE(\MB(C, \Sl_r)) = IE_{\mathrm{st}}(\MB(C, \PGl_r)).\]

\begin{thm}\label{thm:perversetopologicalmirrorsymmetryrank2}
\cref{conj:pervtopmirr} holds for $r=2$.
\end{thm}
\begin{proof}
By Relative Hard Lefschetz, it is enough to show 
\begin{equation*}\label{eq:PIEPIESTR}
   PIE(\MDol(C, \Sl_2))_{\kappa} = PIE(\MDol(C, \Sl_2)_{\gamma}/\Gamma; u,v)(uvq)^{F(\gamma)}
\end{equation*}
for any $\kappa \in \hat{\Gamma}$ and $\gamma = w(\kappa)$.
As in \cref{thm:topologicalmirrorsymmetryrank2}, the case $\kappa=1$, alias $\gamma=0$, is trivial. Suppose then $\kappa \neq 1$ and $\gamma \neq 0$. The perverse filtration on $IH^d(\MDol(C, \Sl_2))_{\kappa}$ is concentrated in degree $d-2g+2$ by \cite[Theorem 5.5]{Mauri20}. Moreover, the Hitchin map $\chi_{\gamma}$ is a quotient by the inverse of the group law of the projection
\[\MDol(C_{\gamma}, \Gl_1)\supset T^*\mathrm{Prym} = \mathrm{Prym} \times \CC^{g-1} \to \CC^{g-1},\]
where $\mathrm{Prym}$ is the connected component of the identity of the kernel of the norm map $\mathrm{Nm}\colon \Pic^0(C_{\gamma}) \to \Pic^0(C)$, given by $\mathrm{Nm}(L)=L \otimes \iota^*L$. Hence, the perverse filtration on $IH^d(\MDol(C, \Sl_2)_{\gamma}/\Gamma)$, with $\gamma \neq 0$, is concentrated in degree $d$; cf proof of \cite[Theorem 6.6]{FelisettiMauri2020}.
Then one easily see that \cref{conj:pervtopmirr} for $r=2$ is equivalent to \cref{thm:topologicalmirrorsymmetryrank2}.
\end{proof}
\textbf{Acknowledgement.}
This work have been supported by the Max Planck Institute for Mathematics.
\bibliographystyle{plain}
\bibliography{construction}

\begin{thebibliography}{10}

\bibitem{BaragliaHekmati17}
D.~Baraglia and P.~Hekmati.
\newblock Arithmetic of singular character varieties and their
  {$E$}-polynomials.
\newblock {\em Proc. Lond. Math. Soc. (3)}, 114(2):293--332, 2017.

\bibitem{deCataldoMaulik2018}
M.~A. {de Cataldo} and D.~{Maulik}.
\newblock The perverse filtration for the {H}itchin fibration is locally
  constant.
\newblock {\em arXiv:1808.02235}, 2018.

\bibitem{deCataldoMigliorini05}
M.~A. de~Cataldo and L.~Migliorini.
\newblock The {H}odge theory of algebraic maps.
\newblock {\em Ann. Sci. \'{E}cole Norm. Sup. (4)}, 38(5):693--750, 2005.

\bibitem{FelisettiMauri2020}
C.~{Felisetti} and M.~{Mauri}.
\newblock {P=W conjectures for character varieties with symplectic resolution}.
\newblock {\em arXiv:2006.08752}, 2020.

\bibitem{GottscheSoergel93}
L.~G\"{o}ttsche and W.~Soergel.
\newblock Perverse sheaves and the cohomology of {H}ilbert schemes of smooth
  algebraic surfaces.
\newblock {\em Math. Ann.}, 296(2):235--245, 1993.

\bibitem{GWZ20}
M.~Groechenig, D.~Wyss, and P.~Ziegler.
\newblock Mirror symmetry for moduli spaces of {H}iggs bundles via p-adic
  integration.
\newblock {\em Invent. Math.}, 221(2):505--596, 2020.

\bibitem{Hausel13}
T.~Hausel.
\newblock Global topology of the {H}itchin system.
\newblock In {\em Handbook of moduli. {V}ol. {II}}, volume~25 of {\em Adv.
  Lect. Math. (ALM)}, pages 29--69. Int. Press, Somerville, MA, 2013.

\bibitem{HauselThaddeus03}
T.~Hausel and M.~Thaddeus.
\newblock Mirror symmetry, {L}anglands duality, and the {H}itchin system.
\newblock {\em Invent. Math.}, 153(1):197--229, 2003.

\bibitem{MartinezMunoz16}
J.~Mart\'{\i}nez and V.~Mu\~{n}oz.
\newblock E-polynomials of the {${\rm SL}(2,\Bbb C)$}-character varieties of
  surface groups.
\newblock {\em Int. Math. Res. Not. IMRN}, (3):926--961, 2016.

\bibitem{MaulikShen2020II}
D.~{Maulik} and J.~{Shen}.
\newblock Cohomological $\chi$-independence for moduli of one-dimensional
  sheaves and moduli of {H}iggs bundles.
\newblock {\em arXiv:2012.06627}, 2020.

\bibitem{MaulikShen2020I}
D.~{Maulik} and J.~{Shen}.
\newblock {Endoscopic decompositions and the Hausel--Thaddeus conjecture}.
\newblock {\em arXiv:2008.08520}, 2020.

\bibitem{Mauri20}
M.~Mauri.
\newblock Intersection cohomology of rank two character varieties of surface
  groups.
\newblock {\em arXiv:2101}, 2021.

\bibitem{Simpson1994}
C.~T. Simpson.
\newblock Moduli of representations of the fundamental group of a smooth
  projective variety. {II}.
\newblock {\em Inst. Hautes \'Etudes Sci. Publ. Math.}, (80):5--79 (1995),
  1994.

\end{thebibliography}
\end{document}